\documentclass[10pt]{amsart}
\usepackage{amsmath}
\usepackage{amssymb}
\usepackage{amsthm}
\usepackage{eepic,epic}
\usepackage{epsfig}

\newlength{\defbaselineskip}
\newcommand{\setlinespacing}[1]%
           {\setlength{\baselineskip}{#1 \defbaselineskip}}

\numberwithin{equation}{section}

\newtheorem{thm}{Theorem}[section]

\newtheorem{lem}[thm]{Lemma}
\newtheorem{prop}[thm]{Proposition}
\theoremstyle{definition}

\theoremstyle{remark}

\numberwithin{equation}{section}

\newcommand{\schr}{e^{it\Delta}}
\newcommand{\schrd}{e^{i(t-s)\Delta}}

\newcommand \mix[2] {L^{#1}_{t}L^{#2}_{x}}
\newcommand \mixx[2] {L^{#1}_{x}L^{#2}_{t}}

\newcommand \mi[2] {L^{#1}_{x}L^{#2}_{t}}

\newcommand \wtp[1] {\widetilde #1'}

\newcommand \mixq   {{\mixx rq}}

\newcommand \mip {{\mi {\wtp r}{\wtp q}}}
\newcommand \miq {{\mi rq}}
\newcommand \wtr {\wtp r}
\newcommand \wtq {\wtp q}

\begin{document}

\title[Inhomogeneous Strichartz estimates]
{On inhomogeneous Strichartz estimates for the Schr\"odinger
equation}
\author{Sanghyuk Lee}
\author{Ihyeok Seo}


\address{Department of Mathematical Sciences, Seoul National University, Seoul 151-747, Republic of Korea}
\email{shklee@snu.ac.kr}
\address{School of Mathematics, Korea Institute for Advanced Study, Seoul 130-722, Republic of Korea}
\email{ihseo@kias.re.kr}

\maketitle

\begin{abstract}
In this paper  we consider inhomogeneous Strichartz estimates in the
mixed norm spaces which are given by taking temporal integration
before spatial integration. We obtain some new estimates, and
discuss about the necessary conditions.
\end{abstract}

\section{Introduction}

To begin with, let us consider the Cauchy problem
\begin{equation*}
\left\{
\begin{array}{ll}
iu_t+\Delta u=F(x,t),\quad (x,t)\in\mathbb{R}^n\times\mathbb{R},\\
u(x,0)=f(x).
\end{array}\right.
\end{equation*}
By Duhamel's principle we have the solution
\[u(x,t)=\schr f(x)-i\int_0^t \schrd F(s) ds.\]
Here $\schr$ is the free propagator which is given by
\[\schr f(x)=(2\pi)^{-n}\int_{\mathbb R^n} e^{i(x\cdot\xi-t|\xi|^2)} \widehat f(\xi) d\xi.\]
The  estimates for the solution in terms of $f$ and $F$ play important roles
in the study of nonlinear Schr\"odinger equations ({\it cf. \cite{C,T}}).
The control of solution $u$ actually consists of two parts, homogeneous $(F=0)$
and inhomogeneous $(f=0)$ part.

It is well known that the homogeneous Strichartz estimate
\begin{equation}\label{homo}
\big\|\schr f\big\|_{L_t^qL_x^r}\le C\|f\|_2
\end{equation}
holds if and only if $2/q=n(1/2-1/r)$, $q\ge 2$ and $(q,r,n)\neq
(2,\infty,2)$ (see \cite{GV,KT} and references therein). But
determining the optimal range of $(q,r)$ and $(\wtq,\wtr)$ for which
the inhomogeneous Strichartz estimate
\begin{equation}\label{inhomo}
\bigg\| \int_0^t \schrd F(s) ds\bigg\|_{L_t^qL_x^r}\le C\|F\|_{L_t^{\wtp q}L_x^{\wtp r}}
\end{equation}
holds is not completed yet. By duality the homogeneous estimates
imply some inhomogeneous estimates but it was observed that the
estimate~\eqref{inhomo} is valid on a wider range than what is given
by admissible pairs $(q,r)$, $(\wtq,\wtr)$ for the homogeneous
estimates~\eqref{homo} (see \cite{CW},\,\cite{K}). Foschi and
Vilela in their independent works (\cite{F2},\cite{Vi}) obtained the
currently best known range of $(q,r)$ and $(\wtq,\wtr)$ for which
\eqref{inhomo} holds. However, there still remain some gaps between
their range and the known necessary conditions. Also, see \cite{LS2}
for a new necessary condition and some weak endpoint estimates.

\subsection{Time-space estimates}
We now consider estimates in different mixed norms which are given
by taking time integration before spatial integration. We call
\eqref{homo} and \eqref{inhomo} \emph{space-time estimate}, and by
\emph{time-space estimate} we mean the estimate  given in
$L^r_xL^q_t$ norms; e.g. \!\!\eqref{tshomo},\,\eqref{iinhomo}. There
are estimates similar to the space-time estimates \eqref{homo}. More
generally, the estimates
\begin{equation}\label{tshomo}
\big\|e^{it\Delta}f\big\|_{L_x^rL_t^q}\leq C\|f\|_{\dot{H}^s}, \quad
s=n/2-2/q-n/r,
\end{equation}
have been of interest. Here $\dot{H}^s$ denotes the homogeneous
Sobolev space of order $s$. Even though \eqref{homo} and
\eqref{tshomo} have the same scaling, they are of different natures.
Especially, for time-space estimate Galilean invariance is no longer
valid in general. The condition $1/q+(n+1)/r\leq n/2$ is necessary
for \eqref{tshomo} even with frequency localized initial datum  $f$
as it is easily seen by using  Knapp's example. It is currently
conjectured that \eqref{tshomo} holds whenever $1/q+(n+1)/r\leq
n/2$, $2\le q<\infty$. When $n=1$, it is known to be true
\cite{KPV}. In higher dimensions \eqref{tshomo} is verified for
$q,r$ satisfying $1/q+(n+1)/r\leq n/2$, additionally $r>16/5$ when
$n=2$, and $r>2(n+3)/(n+1)$ when $n\geq3$ (\cite{LRV}).  The
estimate \eqref{tshomo} is closely related to the maximal
Schr\"odinger estimate  which has been studied to obtain almost
everywhere convergence to initial data. See \cite{Car, DK, Sj, V,
KPV,  LRV, Ro} and references therein for further discussions and
related issues. Also see \cite{L, B} for recent results.

\medskip

In this paper we aim to look for the optimal range of $(\widetilde{r}',r)$
for which the time-space inhomogeneous Strichartz estimate
\begin{equation}\label{iinhomo}
\bigg\| \int_0^t \schrd F(s) ds\bigg\|_{\miq}\le C\|F\|_\mip
\end{equation}
holds for some $q,\wtq$. Obviously, this is weaker than \eqref{inhomo}
if $q\le r$ and $\wtp q\geq\wtp r$ since one can get
\eqref{iinhomo} from \eqref{inhomo} via Minkowski's inequality.
However, as it turns out, the range for \eqref{iinhomo} is quite
different from that of \eqref{inhomo}. The currently known range of
$(1/\wtr,1/r)$ for which \eqref{inhomo} is valid  for some $q, \wtq$
is contained in the closed pentagon with vertices $(1/2,1/2)$, $C'$,
$S'$, $S$, $C$ (see Figure 1) and it is known that \eqref{inhomo}
fails unless $(1/\wtr,1/r)$ is contained in the closed pentagon with
vertices $(1/2,1/2)$, $C'$, $R'$, $R$,  $C$. We will show that
\eqref{iinhomo} is possible only if $(1/\wtr,1/r)$ is contained in
the closed trapezoid $B,$ $R,$ $R'$, $B'$ from which the points
$R,R'$ are removed. In \cite{F}  it was shown  that if $1\le
\wtr\leq2\leq r\le \infty$ and $|1/r+1/{\wtr}-1|<1/n$,  there are
$q, \wtq$ which allow the time delayed estimates in time-space norm.
But in  contrast to the space-time estimate \eqref{inhomo} the above
discussion  shows that mere existence of such $q, \wtq$ for time
delayed estimate is not enough to obtain \eqref{iinhomo} and
accurate information on the possible range of $q, \wtq$ is
important.

\smallskip

 To show \eqref{iinhomo} we work on Fourier transform side
by making use of the fact that the Duhamel part is similar to
multiplier of negative order (see \cite{CKLS, LS}). This allows us
to take advantage of localization in Fourier transform side which
plays important roles in our argument. We believe that this method
is more flexible than the conventional argument which heavily relies
on the dispersive estimate.


\begin{figure}[t]
\label{figure} \centering \setlength{\unitlength}{0.00027in}
\begingroup\makeatletter\ifx\SetFigFont\undefined%
\gdef\SetFigFont#1#2#3#4#5{%
  \reset@font\fontsize{#1}{#2pt}%
  \fontfamily{#3}\fontseries{#4}\fontshape{#5}%
  \selectfont}%
\fi\endgroup%
{\renewcommand{\dashlinestretch}{30}
\begin{picture}(11000,11428)(0,-10)
\path(600,601)(10200,601)(10200,10201)(600,10201)(600,601)
\path(10200,601)(11325,601)(11400,601)
\path(11100.000,526.000)(11400.000,601.000)(11100.000,676.000)
\path(600,10201)(600,11401)
\path(675.000,11101.000)(600.000,11401.000)(525.000,11101.000)
\dottedline{100}(600,10201)(10200,601)
\dottedline{100}(600,5401)(10200,601) \dottedline{100}(5400,10201)(10200,601)
\dottedline{100}(600,5401)(5400,601) \dottedline{100}(5400,10201)(10200,5401)
\path(3000,3001)(7800,7801)
\dottedline{100}(600,10201)(3000,3001) \dottedline{100}(600,10201)(7800,7801)
\path(2520,4441)(2520,5721) \path(5080,8281)(6360,8281)
\path(2200,5401)(5400,8601)
\path(2200,5401)(3000,3001) \path(5400,8601)(7800,7801)
\path(2520,4441)(3000,3001) \path(6360,8281)(7800,7801)
\put(-300,11026){\makebox(0,0)[lb]{\smash{{{\SetFigFont{10}{12.0}{\rmdefault}{\mddefault}{\updefault}$\frac1r$}}}}}
\put(-750,10051){\makebox(0,0)[lb]{\smash{{{\SetFigFont{10}{12.0}{\rmdefault}{\mddefault}{\updefault}$(\frac12,\frac12)$}}}}}
\put(10950,76){\makebox(0,0)[lb]{\smash{{{\SetFigFont{10}{12.0}{\rmdefault}{\mddefault}{\updefault}$\frac1{\widetilde{r}'}$}}}}}
\put(9600,100){\makebox(0,0)[lb]{\smash{{{\SetFigFont{10}{12.0}{\rmdefault}{\mddefault}{\updefault}$(1,0)$}}}}}
\put(-250,100){\makebox(0,0)[lb]{\smash{{{\SetFigFont{10}{12.0}{\rmdefault}{\mddefault}{\updefault}($\frac12$,0)}}}}}
\put(1730,5301){\makebox(0,0)[lb]{\smash{{{\SetFigFont{10}{12.0}{\rmdefault}{\mddefault}{\updefault}$B$}}}}}
\put(5350,8731){\makebox(0,0)[lb]{\smash{{{\SetFigFont{10}{12.0}{\rmdefault}{\mddefault}{\updefault}$B'$}}}}}
\put(-2500,5300){\makebox(0,0)[lb]{\smash{{{\SetFigFont{10}{12.0}{\rmdefault}{\mddefault}{\updefault}$(\frac12,\frac{n-2}{2n})=C$}}}}}
\put(5000,10550){\makebox(0,0)[lb]{\smash{{{\SetFigFont{10}{12.0}{\rmdefault}{\mddefault}{\updefault}$C'=(\frac{n+2}{2n},\frac12)$}}}}}
\put(5100,100){\makebox(0,0)[lb]{\smash{{{\SetFigFont{10}{12.0}{\rmdefault}{\mddefault}{\updefault}$(\frac{n-1}{n},0)$}}}}}
\put(10350,5300){\makebox(0,0)[lb]{\smash{{{\SetFigFont{10}{12.0}{\rmdefault}{\mddefault}{\updefault}$(1,\frac1n)$}}}}}
\put(2350,5850){\makebox(0,0)[lb]{\smash{{{\SetFigFont{10}{12.0}{\rmdefault}{\mddefault}{\updefault}$P$}}}}}
\put(4620,8150){\makebox(0,0)[lb]{\smash{{{\SetFigFont{10}{12.0}{\rmdefault}{\mddefault}{\updefault}$P'$}}}}}
\put(2100,4100){\makebox(0,0)[lb]{\smash{{{\SetFigFont{10}{12.0}{\rmdefault}{\mddefault}{\updefault}$Q$}}}}}
\put(6360,8380){\makebox(0,0)[lb]{\smash{{{\SetFigFont{10}{12.0}{\rmdefault}{\mddefault}{\updefault}$Q'$}}}}}
\put(2570,2580){\makebox(0,0)[lb]{\smash{{{\SetFigFont{10}{12.0}{\rmdefault}{\mddefault}{\updefault}$R$}}}}}
\put(7850,7950){\makebox(0,0)[lb]{\smash{{{\SetFigFont{10}{12.0}{\rmdefault}{\mddefault}{\updefault}$R'$}}}}}
\put(3610,3320){\makebox(0,0)[lb]{\smash{{{\SetFigFont{10}{12.0}{\rmdefault}{\mddefault}{\updefault}$S$}}}}}
\put(7170,6900){\makebox(0,0)[lb]{\smash{{{\SetFigFont{10}{12.0}{\rmdefault}{\mddefault}{\updefault}$S'$}}}}}
\end{picture}
}
\caption{\small The points $B$, $C$, $P$, $Q$, $R$, $S$ and the dual points $B'$, $C'$, $P'$, $Q'$, $R'$, $S'$
when $n\ge 3$.}
\end{figure}
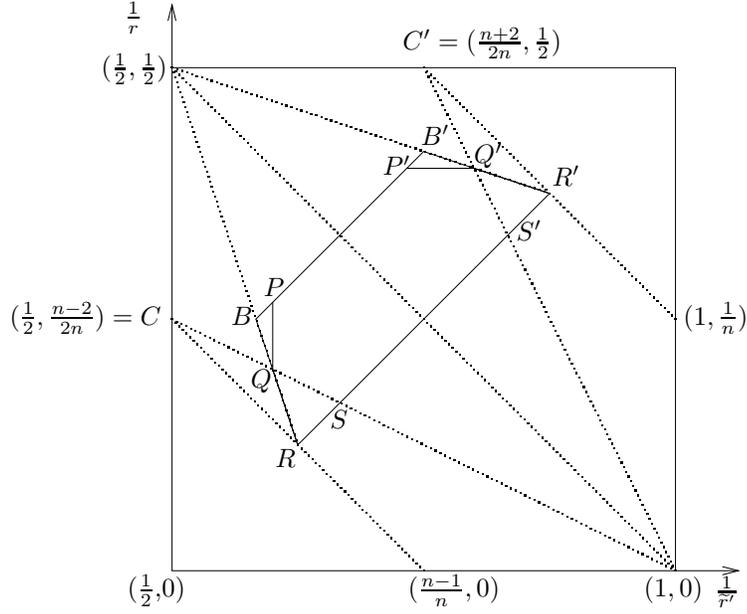


\subsubsection*{\textbf{Necessary conditions}}
We now discuss the conditions on $(q,r)$ and $(\wtq,\wtr)$ which are
necessary for \eqref{iinhomo}. By scaling the condition
\begin{equation}\label{scale}
\frac1{\wtp q}-\frac 1q +\frac n2\left( \frac1{\wtp r}-\frac1{r}\right)=1
\end{equation}
should be satisfied. Using the examples in \cite{F2, Vi}, we see
that the conditions which are needed for \eqref{inhomo} are also
necessary for \eqref{iinhomo}:

\begin{align}\label{x-cond} \wtr< 2< r,
\quad \frac1\wtr-\frac1r\le \frac2n,\quad 1-\frac1n\le \frac1\wtr+\frac1r\le 1+\frac1n,
\\
\label{t-cond} \wtq\le q,\quad \frac1q<n(\frac12-\frac1r),\quad\frac1\wtq>1-n(\frac1\wtr-\frac12).
\end{align}
By considering  additional test functions, we get the following
conditions which will be shown later (see Section \ref{secsec4}):
\begin{align}
\label{q2}
\frac1{\widetilde{q}'}-\frac1q+(n+1)(\frac1{\widetilde{r}'}-\frac1r)\geq2,\\
\label{q3}
\frac1{\widetilde{q}'}-\frac1q\geq\frac{2n}r-n+1,\quad
\frac1{\widetilde{q}'}-\frac1q\geq n+1-\frac{2n}{\widetilde{r}'}.
\end{align}

To facilitate the statement of our results, for $n\ge 3$, let us
define points $B,$ $C,$ $P,$ $Q,$ $R,$ and $S$ which are contained
in $[\frac12,1]\times [0,\frac12]$ by setting
\begin{align*}
&B=(\frac{n+3}{2(n+2)},\frac{n-1}{2(n+2)}),\,\,\,C=(\frac12,\frac{n-2}{2n}),\,\,\, P=(\frac{n+2}{2(n+1)},\frac{n^2}{2(n+1)(n+2)}),\\
&Q=(\frac{n+2}{2(n+1)},\frac{n-2}{2(n+1)}),\,\,\,R=(\frac{n+1}{2n},\frac{n-3}{2n}),\,\,\,S=(\frac{n}{2(n-1)}, \frac{(n-2)^2}{2n(n-1)}),
\end{align*}
and we also define the dual points $B',$ $C'$, $P'$, $Q'$, $R'$, $S'$ by setting $X'=(1-b,1-a)$ when $X=(a,b)$.
(See Figure 1.)
Let $\mathcal N(n)$ be the closed trapezoid with vertices $B,$ $B'$, $R$, $R'$
from which the points $R,R'$ are removed.
Being combined with \eqref{scale}, \eqref{q2} gives
\begin{equation}\label{rcon1}
\frac 1\wtr -\frac1r\ge \frac{2}{n+2},
\end{equation}
and the first and second conditions in \eqref{q3} give
\begin{equation}
\label{rcon2}
n\big(1-\frac1{2\wtr}\big)\ge \frac{3n}{2r}, \quad \frac{3n}{2\wtr}\ge n\big(1-\frac{1}{2r}\big),
\end{equation} respectively.
Also, by \eqref{scale} and \eqref{t-cond}, we see that  $(1/\wtr,1/r)\neq R,R'$.
Hence, from  this, \eqref{x-cond}, \eqref{rcon1} and \eqref{rcon2},
it follows that \eqref{iinhomo} holds only if $(1/\wtr,1/r)\in \mathcal N(n)$.

\subsubsection*{\textbf{Sufficiency part}}
We will show a stronger  estimate
\begin{equation}\label{stro}
\bigg\| \int_{-\infty}^t \schrd F(s) ds\bigg\|_{\miq}\le C\|F\|_\mip,
\end{equation}
which implies \eqref{iinhomo} and $\| \int_{-\infty}^\infty\schrd
F(s) ds\|_{\miq}\le C\|F\|_\mip.$ As mentioned above, if $q\le r$
and $\wtp q\geq\wtp r$, from the known range of the space-time
estimate (\cite{F2,V}), one can get \eqref{stro} for $(1/\wtr,1/r)$
contained in the closed hexagon $\mathcal H$ with vertices $P$, $Q$,
$S$, $P'$, $Q'$, $S'$ from which  the line segments $[P,Q],[P',Q']$
and the points $S,S'$ are removed\,\footnote{\,In fact, when $q\le r$
and $\wtp q\geq\wtp r$, \eqref{scale} and \eqref{t-cond} are
satisfied   if $(1/\wtr,1/r)$ is contained in $\mathcal H$. So we
can use the known space-time estimate.}. We extend the range further
to include the triangular region $\Delta QRS$ and $\Delta Q'R'S'$.
It should be noted that no inhomogeneous space-time estimate
\eqref{inhomo} is known for $(1/\wtr,1/r)$ which is contained in the
interior of $\Delta QRS$ and $\Delta Q'R'S'$.

\begin{thm}\label{t-s thm}
Let $n\ge3$ and $\mathcal S(n)$ be the open hexagon with vertices $P,$ $Q,$ $R$, $P'$, $Q'$, $R'$ to which
the line segments $(P,P')$ and $(R,R')$ are added.
If $(1/\wtr,1/r)\in \mathcal S(n)$, then \eqref{stro} holds for some $q,\wtq$.
\end{thm}

For $(1/\wtr, 1/r)$  contained in the region $\Delta
QRS\setminus[Q,R]$, the estimate \eqref{stro} is available if
$(\wtq,q)$ satisfies \eqref{scale}, \eqref{t-cond} and
additionally $\frac1q< n(\frac1{\widetilde{r}'}-\frac12),$
$\frac1{\widetilde{q}'}> 1-n(\frac12-\frac1r).$ Being combined
with \eqref{scale}, these additional conditions are due to the third
inequality of \eqref{locp} and its dual one. By duality the same
holds for $(1/\wtr, 1/r)$ which is contained in the region $\Delta
Q'R'S'\setminus [Q',R']$. Making use of the currently known
time-space homogeneous estimates \eqref{tshomo} ({\it cf. \cite{LRS,
LRV}}) together with the argument of this paper, it is possible to
obtain further estimates on a lager range of $q,\wtq$ but these
estimates are not enough to extend the range of $(\wtr,r)$.

\medskip

When $n=2$, \eqref{stro} holds if $(1/\wtr,1/r)$  is contained in
the open pentagon with vertices $P,$ $Q$, $(1,0)$, $Q'$, $P'$ to
which the line segment  $(P, P')$ is added but it is not new.
This just follows from the known range of the space-time estimate (\cite{F2, Vi}).
When $n=1$, it is possible to obtain the full range except some endpoint estimates.
In fact, from the necessary conditions, \eqref{iinhomo}
is possible only if $(1/\wtr,1/r)$ is contained in the closed
triangle $\Delta$ with vertices $(\frac23, 0),$ $(1,0)$,
$(1,\frac13)$.

\begin{thm}\label{thm1d}
Let $n=1$. Then \eqref{stro} holds for some $q$, $\wtq$ provided that $(1/\wtr,1/r)$ is contained in
$\Delta\setminus\big([(\frac23,0),(1,0)]\cup[(1,\frac13),(1,0)]\big)$.
In fact, \eqref{stro} holds if $q$, $\wtq$  satisfies $1<\wtq<2<q<\infty$ and
$1/{\wtr}-1/{r}+1/2\wtq-1/2q\ge1$.
\end{thm}

The rest of this paper is organized as follows: In Section \ref{secsec2}
we obtain some frequency localized estimates which
will be used in later sections. Then, using these estimates and a
summation method, we prove Theorem \ref{t-s thm} and \ref{thm1d} in Section \ref{secsec3}.
Nextly, we show the necessary conditions \eqref{q2}
and \eqref{q3} in Section \ref{secsec4}.

\medskip

Throughout this paper, the letter C stands for a constant which is
possibly different at each occurrence. In addition to the symbol
$~\widehat~~$, we use $\mathcal F(\cdot)$ to denote the Fourier
transform, and $\mathcal F^{-1}(\cdot)$ to denote the
inverse Fourier transform. Finally, we denote by $\chi_E$ the
characteristic function of a set $E$.


\section{Preliminaries}\label{secsec2}

In this section we prove several preliminary estimates which will be used for the
proof of Theorem~\ref{t-s thm}, which is to be shown in Section \ref{secsec3}.

Let us define the operator $T_\delta$
for dyadic numbers $\delta\in2^{\mathbb{Z}}:=\{2^z:z\in\mathbb{Z}\}$ by
\begin{equation}\label{t-del}
T_\delta F=\int \delta\phi(\delta(t-s))\schrd F(s) ds
\end{equation}
where $\phi$ is a smooth function supported in $(1/2,2)$ such that
$\sum_{k=-\infty}^\infty \phi(2^k t)=1,$ $ t>0.$
Then we may write
\begin{equation}\label{dyadic}
TF:=\int_{-\infty}^t \schrd F(s) ds=\sum_{\delta\in2^{\mathbb{Z}}} \delta^{-1}T_\delta F.
\end{equation}
By direct computation it is easy to see that
\begin{equation}\label{compu}
\widehat {T_\delta F}(\xi,\tau)=\widehat\phi(\frac{\tau+|\xi|^2}\delta) \widehat F(\xi,\tau) .
\end{equation}
By this dyadic decomposition in time, the boundedness problem for $T$
is essentially reduced to obtaining suitable bounds for $T_\delta$ in terms of $\delta$.
From this one may view    the operator $F\to \int_{-\infty}^t \schrd F(s) ds$ as
the multiplier operator of negative order $1$ which is associated to the paraboloid.

\begin{prop}\label{radial} Let $n\ge 2$.
Suppose that  Fourier transform of $F$ is supported in
$\{(\xi,\tau)\in\mathbb{R}^n\times \mathbb R: 1/2\le|\xi|\le 2\}$.
Then we have
\[
\|T_\delta F\|_{\mi{r}{2}}\le C\delta^{-\frac {n-1}2+\frac n\wtr}\|F\|_{\mi{\wtr}{2}}
\]
for $r,\wtr$ satisfying $1\le \wtr\le 2$ and  $(n+1)/r\le (n-1)(1-1/\wtr)$.
\end{prop}

\begin{proof}
In view of interpolation it is enough to consider the cases
$(\wtr, r)=(2,\frac{2(n+1)}{n-1})$ and $(1,\infty)$.
This actually gives the estimates along the line $(n+1)/r=(n-1)(1-1/\wtr)$.
The other estimates follow from  Bernstein's inequality
because the spatial Fourier transform of $F$ is compactly supported.

\smallskip

\noindent\textit{The case $(\wtr, r)=(2,\frac{2(n+1)}{n-1})$.}
By duality it is enough to show that
\begin{equation*}
\|T_\delta F \|_{\mi 22}\le C\delta^\frac12\|F\|_{\mi {\frac{2n+2}{n+3}}2}.
\end{equation*}
Since $\widehat F(\cdot, \tau)$ is supported in $\{|\xi|\sim 1\}$,
by \eqref{compu} and Plancherel's theorem we have
\[\|T_\delta F \|_{\mi 22}^2\le C\iint_{1/2\leq|\xi|\leq2}
\big|\widehat\phi(\frac{\tau+|\xi|^2}\delta) \widehat F(\xi,\tau)\big|^2 d\xi d\tau.\]
So we are reduced to showing that
\begin{equation}\label{ra}
\iint_{1/2\leq|\xi|\leq2}  \big|\widehat\phi(\frac{\tau+|\xi\big|^2}\delta) \widehat
F(\xi,\tau)\big|^2 d\xi d\tau\le C\delta \|F\|_{\mi {\frac{2n+2}{n+3}}2}^2.
\end{equation}
Then the left hand side equals to
\[\iint_{\frac12}^2 \big|\widehat\phi(\frac{\tau+r^2}\delta)\big|^2
\int_{S^{n-1}} |\mathcal F_x\mathcal F_t F(r\theta,\tau)|^2
d\theta\, r^{n-1}dr d\tau.\] Using Tomas-Stein theorem \cite{St}
($L^2$-restriction estimate  to the sphere $rS^{n-1}$, $r\sim 1$),
we see that
$$\int_{S^{n-1}} |\mathcal F_x\mathcal F_t F(r\theta,\tau)|^2 d\theta \le
C\|\mathcal F_t F(\cdot,\tau)\|_{L_x^{\frac{2n+2}{n+3}}}^2.$$
Taking integration in $r$, it follows that
\[\iint_{1/2\leq|\xi|\leq2}  \big|\widehat\phi(\frac{\tau+|\xi|^2}\delta) \widehat
F(\xi,\tau)\big|^2 d\xi d\tau\le C\delta \|\mathcal F_t F(\tau)\|_{L^2_\tau L^{\frac{2n+2}{n+3}}_x}^2.\]
By Minkowski's inequality and Plancherel's theorem, we get \eqref{ra}.

\smallskip

\noindent\textit{The case $(\wtr, r)=(1,\infty)$.}
Note that $T_\delta F$ can be written as
\[ T_\delta F(x,t)=\iint K_\delta(x-y,t-s) F(y,s) dyds,\]
where
\begin{align}
K_\delta(y,s)&=\delta\phi(\delta s) \iint e^{i(-sr^2+ry\cdot\theta)} \psi(r) d\theta dr\label{kernel}\\
&=\iiint
\widehat\phi(\frac{\tau+r^2}\delta)e^{i(s\tau+ry\cdot\theta)}
 \psi(r) d\theta dr d\tau \nonumber
\end{align}
and $\psi\in C_0^\infty(1/2,2)$. Since $|K_\delta(y,s)|\leq C|\delta\phi(\delta s)|$, by Young's inequality, we have
$\|T_\delta F\|_{L_x^\infty L_t^2}\leq C\|F\|_{L_x^1L_t^2}$. So we  may assume that $\delta\lesssim1$.

By the choice of $\phi$, $K_\delta\neq 0$ for $s\sim \delta^{-1}$.
Hence by non-stationary phase method, we see that
$|K_\delta(y,s)|\le C\delta^{M}$ if $|y|\le \delta^{-1}/100$
and $|K_\delta(y,s)|\le C(1+|y|)^{-N}$ if $|y|\ge 100\delta^{-1}.$
Let us set
$\chi_\delta(y)=\chi_{\{\delta^{-1}/100\le |y|\le 100\delta^{-1}\}}$,
$\widetilde K_\delta(y,s)=K_\delta(y,s)\chi_\delta(y)$,
and
\[ \widetilde T_\delta F(x,t)=\iint \widetilde K_\delta(x-y,t-s)F(y,s)dyds.\]
Then it is enough to show that
\begin{equation}\label{ra2}
\|\widetilde T_\delta F\|_{\mi \infty 2}\le C\delta^\frac{n+1}2\| F\|_{\mi 12}.
\end{equation}
Form  \eqref{kernel}, it follows that
\begin{align*}
\mathcal F_t(\widetilde T_\delta F)&(x,\tau)\\
=&\int\mathcal F_tF(y,\tau)\int\widehat\phi(\frac{\tau+r^2}\delta)
\chi_\delta(x-y)\psi(r)\left(\int_{S^{n-1}} e^{ir(x-y)\cdot\theta}d\theta \right)drdy.
\end{align*}
Hence by plancherel's theorem we see that
$\|\widetilde{T}_\delta F(x,\cdot)\|_{L^2_t}^2$ is bounded by
$$\int \left[\int|\mathcal F_t F(y,\tau)|\int\Big|\widehat\phi(\frac{\tau+r^2}\delta)\chi_\delta(x-y)\psi(r)
\int_{S^{n-1}} e^{ir(x-y)\cdot\theta}d\theta\Big| drdy\right]^2 d\tau.$$
By using the fact that
$\int e^{ir x\cdot\theta} d\theta=O(|x|^{-\frac{n-1}2})$
for large $|x|$, and taking integration in $r$,
\begin{align*}
\|\widetilde{T}_\delta F(x,\cdot)\|_{L^2_t}^2
&\leq
C\delta^{n+1}\int\left(\int |\mathcal F_t F(y,\tau)| dy\right)^2 d\tau.
\end{align*}
By Minkowski's inequality and Plancherel's theorem,
$$\| \mathcal F_t F \|_{L^2_\tau L^1_x}\le \| \mathcal F_t F \|_{L^1_x L^2_\tau }=\|  F \|_{L^1_x L^2_t}.$$
Hence  we get~\eqref{ra2}.
\end{proof}

Throughout this paper we will use the following summation lemma several times which
is due to Bourgain \cite{Bour}  (Also see \cite{CSWW}  for a generalization.)
The lemma is a  version of Lemma 2.3 in \cite{LS} for Banach-valued functions.
(For a proof we refer the reader to \cite{LS}.)

\begin{lem}\label{summation}
Let $\varepsilon_1,\varepsilon_2>0$.  Let $1\le q\le \infty$ and
$1\le r_1, r_2<\infty$.  Suppose that $f_1(y,z),f_2(y,z),$
$...,f_j(y,z),...$ be functions defined on $\mathbb R^l\times
\mathbb R^m$ which satisfies that $\|f_j\|_{L^{r_1}_yL^{q}_z}\le
M_12^{\varepsilon_1j}$ and $\|f_j\|_{L^{r_2}_yL^{q}_z}\le
M_22^{-\varepsilon_2j}$ for $1\leq r_1, r_2<\infty$. Then
\[\|\sum f_j\|_{L^{r,\infty}_yL^{q}_z}\le CM_1^{\theta}M_2^{1-\theta},\]
where $\theta=\varepsilon_2/(\varepsilon_1+\varepsilon_2)$ and
$1/r=\theta/{r_1}+(1-\theta)/{r_2}$.  Here we denote by $L^{r,\infty}_y$ the
weak $L^r$ space.
\end{lem}

Using this lemma, we remove the assumption that the spatial Fourier
transform of $F$ is supported in $\{|\xi|\sim 1\}$.

\begin{prop}\label{basic} Let $n\ge 3$.
Suppose that the spatial Fourier transform of $F$ is supported in $\{(\xi,\tau)\in\mathbb{R}^n\times \mathbb R: |\xi|\le 2\}$.
Then we have
\[\|T_\delta F\|_{\mi{r}{2}}\le C\delta^{-\frac {n-1}2+\frac n\wtr}\|F\|_{\mi{\wtr}{2}}\]
for $r,\wtr$ satisfying $n/(n-1)< \wtr\le 2$ and $1/r+1/\wtr\le(n-1)/n$.
\end{prop}

\begin{proof} Since we are assuming that  $F$ is supported in $\{(\xi,\tau): |\xi| \le
2\}$, we may break $T_\delta$ so that
\[T_\delta=\sum_{j\ge -1}T_\delta^j,\]
where $T_\delta^j$ is given by
\[\widehat {T_\delta^j F}(\xi,\tau)=\widehat\phi(\frac{\tau+|\xi|^2}\delta)
\phi(2^j|\xi| )\widehat F(\xi,\tau).\] From rescaling we have
\[ T_\delta^j F(x,t)=T_{2^{2j}\delta} F_j(2^{-j}x, 2^{-2j}t),\]
where $F_j=\phi(|D|)F(2^{j}\cdot, 2^{2j}\cdot)$.
Thus, by Proposition \ref{radial} we see that
\begin{equation}\label{p21}
\|T_\delta^j F\|_{\mi{r}{2}}\le
C\delta^{-\frac {n-1}2+\frac n\wtr}2^{jn(\frac{1}r+\frac1\wtr-\frac{n-1}n)}\|F\|_{\mi{\wtr}{2}}
\end{equation}
for $r,\wtr$ satisfying $1\le \wtr\le 2$ and $(n+1)/r\le
(n-1)(1-1/\wtr)$. If $1/r+1/\wtr<(n-1)/n$, we can sum to get  the
desired estimate. To the estimates for  the endpoint cases $1/r+1/\wtr=(n-1)/n$, we use
Lemma \ref{summation}.

Fix $\wtr,r$ such that  $1/r+1/\wtr=(n-1)/n$ and $n/(n-1)< \wtr<2$.
We now choose $r_1, r_2,$ so that  $(n+1)/r_i\le (n-1)(1-1/\wtr)$,
$i=1,2,$ and
$$\frac1{r_2}+\frac1\wtr<\frac{n-1}n<\frac1{r_1}+\frac1\wtr.$$
Note that $(1/\widetilde{r}',1/r)$ is on the open segment joining
$(1/\widetilde{r}',1/r_1)$ and $(1/\widetilde{r}',1/r_2)$. From
\eqref{p21} we see
\begin{equation*}
\|T_\delta^j F\|_{\mi{r_i}{2}}\le
C\delta^{-\frac {n-1}2+\frac n\wtr}2^{jn(\frac1{r_i}+\frac1\wtr-\frac{n-1}n)}\|F\|_{\mi{\wtr}{2}}
\end{equation*}
for $i=1,2$. We now can apply Lemma \ref{summation} with
$\varepsilon_i=n|\frac1{r_i}+\frac1\wtr-\frac{n-1}n|$. So, we get
\begin{equation*}
\|T_\delta^j F\|_{\mi{r,\infty}{2}}\le
C\delta^{-\frac {n-1}2+\frac n\wtr}\|F\|_{\mi{\wtr}{2}}.
\end{equation*}
This weak type estimate  for $1/r+1/\wtr=(n-1)/n$ and $n/(n-1)<
\wtr<2$ can be strengthened to strong type by real interpolation.
Lastly, the estimate for  $(1/\wtr, 1/r)=(\frac12,\frac{n-2}{2n})$
can be obtained directly from
\begin{equation}\label{qqqqq}
\|T_\delta F\|_{\mix{2}{\frac{2n}{n-2}}}\le C\delta^\frac12 \|F\|_{\mix 22}
\end{equation}
via Minkowski's inequality.
This also follows from the endpoint space-time homogeneous estimate.
Indeed, by H\"older's inequality we see
$$|T_\delta F(x,t)|\leq C\delta^{\frac12}\big\|\schrd F(s)\big\|_{L_s^2},$$
and so
\[\|T_\delta F\|_{\mix{2}{\frac{2n}{n-2}}}\le C
\delta^{\frac12}\big\|\schrd F(s)\big\|_{L_s^2\mix{2}{\frac{2n}{n-2}}}\]
by Minkowski's inequality.
By applying \eqref{homo} with $(q,r)=(2,2n/(n-2))$, we get \eqref{qqqqq}.
\end{proof}


\section{Sufficiency part: Proofs of Theorem \ref{t-s thm} and \ref{thm1d}} \label{secsec3}

In this section we will prove Theorem \ref{t-s thm} and \ref{thm1d}.
We may assume that the space time Fourier transform of
$F$ is supported in the set $\{(\xi,\tau): |\xi| \le 2, \, |\tau|\le
2\}$ since  this additional assumption can be simply removed by
rescaling together with the condition \eqref{scale}.

\subsubsection*{\textbf{Proof of Theorem \ref{t-s thm}}} Since we already have the estimates in the hexagon
$\mathcal{H}$, to show \eqref{stro} it suffices to show the
estimates when $(1/\widetilde{r}',1/r)\in\big(\Delta QRS\cup \Delta
Q'R'S'\big)\setminus\big([Q,R]\cup[Q',R']\big)$. By duality and
complex interpolation, it is enough to show the case where
$(1/\widetilde{r}',1/r)\in\Delta QRS\setminus([Q,R]\cup  [Q,S])$.

Let $\Omega=\Omega(n)$ denote the closed triangle with vertices $C$,
$(\frac{n-1}n,0)$, $(1,0)$ from which the point $(\frac{n-1}n,0)$ is
removed. The proof is then based on the following estimate:

For $(1/\wtr,1/r)\in\Omega$,
\begin{equation}\label{lolo}
\|T_\delta F\|_{\miq}\le C\delta^{\frac1\wtq-\frac1q+\frac n2(\frac1\wtr-\frac1r)}\|F\|_{\mip}
\end{equation}
holds provided that
\begin{equation*}
-\frac n2(\frac1r+\frac1{\widetilde{r}'}-1)\leq\frac1q\leq\frac1{\widetilde{q}'}
\leq 1+\frac n2(\frac1r+\frac1{\widetilde{r}'}-1).
\end{equation*}

This can be shown by interpolating the case
$(1/\widetilde{r}',1/r)=(1,0)$ and the case in which
$(1/\widetilde{r}',1/r)$ is on  the line segment joining $C$ and
$(\frac{n-1}n,0)$. Since Proposition~\ref{basic} already gives the
estimates on the line segment, we need only to show that
\begin{equation*}
\|T_\delta F\|_{L_x^\infty L_t^q}\le C\delta^{\frac1\wtq - \frac1q+\frac n2} \|F\|_{L_x^1L_t^{\widetilde{q}'}}
\end{equation*}
for $1\le \wtq\leq q\le \infty$. By Minkowski's inequality, it is
enough to show that
\begin{equation*}
\|T_\delta F\|_{L_t^qL_x^\infty}\le
\delta^{\frac1\wtq - \frac1q+\frac n2} \|F\|_{L_t^{\widetilde{q}'}L_x^1}.
\end{equation*}
Using the fact that $\big\|e^{i(t-s)\Delta}g\big\|_{L_x^\infty}\leq
C|t-s|^{-n/2}\|g\|_{L_x^1}$ (dispersive estimate), this follows from \eqref{t-del} and Young's inequality.

Now we fix $\wtr,r$ such that $(1/\widetilde{r}',1/r)\in \Delta
QRS\setminus( [Q,S]\cup[Q,R])$. We claim that there is
$(1/\widetilde{q}',1/q)\in[0,1]\times[0,1]$ which satisfies
\eqref{scale} and
\begin{equation}\label{locp} -\frac
n2(\frac1r+\frac1{\widetilde{r}'}-1)<\frac1q\leq\frac1{\widetilde{q}'}
< 1+\frac n2(\frac1r+\frac1{\widetilde{r}'}-1).
\end{equation}
Indeed, since $(1/\widetilde{r}',1/r)\in \Delta QRS\setminus(
[Q,S]\cup[Q,R])$, it follows that
\begin{equation}\label{567}
0\leq1-\frac n2(\frac1{\widetilde{r}'}-\frac1r)<1-n(\frac1{\widetilde{r}'}-\frac12)
<1+\frac n2(\frac1r+\frac1{\widetilde{r}'}-1)<1.
\end{equation}
The third inequality in \eqref{567} says that
$(1/\widetilde{r}',1/r)$ lies above the line joining $Q, R$. Hence,
there exists $1<\wtq<\infty $ such that
\begin{equation}\label{5678}
1-n(\frac1{\widetilde{r}'}-\frac12)<\frac1{\widetilde{q}'}<1+\frac
n2(\frac1r+\frac1{\widetilde{r}'}-1).
\end{equation}
(Note that the first inequality  is also one of the necessary
conditions in \eqref{t-cond}.) Now just set $\frac1q=\frac1{\wtp q}
+\frac n2\left( \frac1{\wtp r}-\frac1{r}\right)-1$.
So, \eqref{scale} is satisfied obviously. Then the first inequality
in \eqref{567} gives the second  in \eqref{locp}, and the first
in \eqref{5678} implies the first in \eqref{locp}.
From \eqref{locp}, we can find a small neighborhood $V$ of
$(1/\widetilde{r}',1/r)$, contained in $\Omega$, such that for
$(1/a,1/b)\in V$
\[-\frac n2(\frac1b+\frac1 a-1)<\frac1q\leq\frac1{\widetilde{q}'} <
1+\frac n2(\frac1b+\frac1{a}-1).\] Therefore, by \eqref{lolo} we
have for $(1/a,1/b)\in V$
\begin{equation}
\label{aa}
\|\delta^{-1}T_\delta F\|_{L_x^{b}L_t^q}\le
C\delta^{\frac1\wtq-\frac1q+\frac n2(\frac1{a}-\frac1{b})-1}
\|F\|_{L_x^{a}L_t^{\widetilde{q}'}}.
\end{equation}

Once this is obtained, we can prove the desired estimates by repeating the argument in the proof of Proposition~\ref{basic}. In fact, we consider a point
$(1/a_0,1/b_0)\in V$ on the line $1/a-1/b=1/\widetilde{r}'-1/r$ and
choose two points $(1/a_0,1/b_i)\in V,\, i=1,2,$ such that
$$\frac1{\widetilde{q}'}-\frac1q+\frac n2(\frac1{a_0}-\frac1{b_1})-1
<0<\frac1{\widetilde{q}'}-\frac1q+\frac
n2(\frac1{a_0}-\frac1{b_2})-1.$$ Then replacing $(a,b)$ with
$(a_0,b_1)$, $(a_0,b_2)$ in  \eqref{aa}, we have two estimates to
which we can  apply Lemma~\ref{summation} with
$\varepsilon_i=|\frac1{\widetilde{q}'}-\frac1q+\frac
n2(\frac1{a_0}-\frac1{b_i})-1|$. Hence, we get
$$\Big\|\sum_{\delta\in2^{\mathbb{Z}}}\delta^{-1}T_\delta F\Big\|_{L_x^{b_0,\infty}L_t^q}\le
C\|F\|_{L_x^{a_0}L_t^{\widetilde{q}'}}$$ for all $(1/a_0,1/b_0)\in V$
if  $1/a-1/b=1/\widetilde{r}'-1/r$. We now interpolate these
estimates to get the strong type, in particular,  at $(1/\wtr,1/r)$.
This completes the proof.

\subsubsection*{\textbf{Proof of Theorem \ref{thm1d}}}

First we claim that for $1\le \wtr,\wtq\le 2\le r,q\le \infty$, and
$0<\delta\ll 1$,
\begin{equation}\label{1d}
\|T_\delta F\|_{\mixq}\le
C\delta^{\frac1{\wtr}-\frac1{r}+\frac12(\frac1\wtq-\frac1q)}\|F\|_{\mixx
{\wtr}{\wtq}}
\end{equation}
whenever $\widehat F$ is supported in $\{(\xi,\tau)\in \mathbb
R\times \mathbb R: |\xi|\le 1, |\tau|\sim  1\}$. From \eqref{compu}
we see that the Fourier transform of $T_\delta$ is essentially
supported in the $\delta$-neighborhood of $\{(\xi,\tau):
\tau=-|\xi|^2, |\tau|\sim 1\}$. Hence it is sufficient to show
\eqref{1d} by assuming that the Fourier support of $F$ is contained in
$\{(\xi,\tau)\in \mathbb R\times \mathbb R: |\xi|\sim 1,
|\tau|\lesssim  1\}$. The contribution from the other region is
negligible.

Under this  assumption, by \eqref{compu},
Plancherel's theorem in $t$, and H\"older's inequality it follows that
\begin{align*}
\|T_\delta F(x,\cdot) \|_{L^2_t}
\le  &C\Big(\int
\Big|\int_{1/2\leq|\xi|\leq2} e^{ix\xi}\,\widehat\phi(\frac{\tau+|\xi|^2}\delta) \widehat F(\xi,\tau)d\xi \Big|^2 d\tau\Big)^\frac12 \\
\le  & C\delta^\frac12\Big(\int
\int_{1/2\leq|\xi|\leq2}\big|\widehat F(\xi,\tau)\big|^2  d\xi  d\tau\Big)^\frac12 .
\end{align*}
Plancherel's theorem gives $\|T_\delta
F\|_{\mixx\infty 2}\le C\delta^\frac12\|F\|_{\mixx 22}$.
By this and duality
we have $\|T_\delta F\|_{\mixx\infty 2}\le C\delta\|F\|_{\mixx 12}$,
and from dispersive estimate $\|T_\delta F\|_{\mixx\infty\infty}\le
C\delta^\frac32\|F\|_{\mixx 11}$. Interpolation between these two
estimates gives for $1\le\wtq\le 2\le q\leq\infty$
\[
\|T_\delta F\|_{\mixx \infty q}\le
C\delta\delta^{\frac12(\frac1\wtq-\frac1q)}\|F\|_{\mixx 1\wtq}.
\]

Let $Q\subset \mathbb R^{1+1}$ be a cube of
side length $\delta^{-1}$ and $\widetilde Q$ be the cube of side
length $C\delta^{-1}$ which has the same center as $Q$. Here $C>0$
is a sufficiently large constant. By H\"older's inequality  we have  for $1\le \wtr,\wtq\le 2\le
r,q\le \infty$, and $0<\delta\ll 1$,
\begin{equation}\label{1dlocal}\|T_\delta (\chi_{\widetilde Q}F)\|_{\mixx rq(Q)}\le C \delta^{\frac1{\wtr}-\frac1{r}+\frac12(\frac1\wtq-\frac1q)}
\| F\|_{\mixx \wtr\wtq}.
\end{equation}
We now deduce \eqref{1d} from this. Firstly,  from the assumption that the Fourier transform of $F$ is contained in
$\{(\xi,\tau)\in \mathbb R\times \mathbb R: |\xi|\sim 1\},$ we observe that  $T_\delta$ is localized at scale
$\delta^{-1}$ in $x$. More precisely, the kernel $K_\delta$
of $T_\delta$ satisfies that
\[ |K_\delta(x,t)|\le C\delta^M
\,\chi_{[1/2\delta,\,2/\delta]}(|t|) (1+|x|)^{-M}
\]
for
any $M$ if $|x|\ge C\delta^{-1}$. (See \eqref{kernel} and the
paragraph below it).  Hence it follows that  if $(x,t)\in Q$, then
\begin{equation}\label{loc} | T_\delta F(x,t)|\le C |\ T_\delta \chi_{\widetilde Q} F(x,t)| + C\delta^{M}(\mathcal E_\delta\ast |F|)(x,t)
\end{equation}
for some large $M>0$ where $\mathcal E_\delta=\chi_{[1/2\delta,\,2/\delta]}(t) (1+|x|)^{-M}$.
Let $\{Q\}$ be a collection of (essentially disjoint) cubes of side length $\delta^{-1}$ which
cover $\mathbb R^{1+1}$. Then by \eqref{loc} we have
\[\| T_\delta F\|_{\mixx rq}\le C \Big\|\sum_{Q}\chi_Q| T_\delta (\chi_{\widetilde Q}F)|\Big\|_{\mixx rq} + C \delta^M \|F\|_{\mixx \wtr\wtq}\]
because $r,q\geq \wtr, \wtq$.
Hence, by Minkowski's inequality and \eqref{1dlocal}  we have
\begin{align*}
\| T_\delta F\|_{\mixx rq}
&\le C\Big(\sum_{Q}\| T_\delta (\chi_{\widetilde Q}F)\|_{\mixx rq(Q)}^p\Big)^\frac1p+ C \delta^M \|F\|_{\mixx \wtr\wtq}\\
& \le C \delta^{\frac1{\wtr}-\frac1{r}+\frac12(\frac1\wtq-\frac1q)} \Big(\sum_{Q}\|\chi_{\widetilde Q}F\|_{\mixx \wtr\wtq}^p\Big)^\frac1p+ C \delta^M \|F\|_{\mixx \wtr\wtq}
\end{align*}
where $p=\min(q,r)$. Since $r,q\geq \wtr, \wtq$, using Minkowski's inequality again,
 we get the desired inequality \eqref{1d}.

\medskip

For $j\in\mathbb{Z}$, let us define the multiplier operators
$\mathcal P_jF$  by
$$\widehat{\mathcal P_jF}(\xi,\tau)=\phi(2^j\tau)\widehat{F}(\xi,\tau).$$
Using \eqref{dyadic}, \eqref{1d}, Lemma \ref{summation}, and
repeating the previous argument, one can show that
\begin{align}\label{ess}
\Big\|\mathcal P_0\Big(\int_{-\infty}^t \schrd F(s)
ds\Big)\Big\|_{\mixx {r}q}
&=\big\|\sum_{\delta\in\mathbb{Z}}\delta^{-1}T_\delta (\mathcal
P_0F)\big\|_{\mixx {r}q}
\\
&\le C\|\mathcal P_0F\|_{\mixx {\wtr}{\wtq}} \nonumber
\end{align}
provided that  $1< \wtr< 2< r<\infty $,  $1\le \wtq< 2<q\le
\infty$, and $\frac1{\wtr}-\frac1{r}+\frac12(\frac1\wtq-\frac1q)\ge
1$. In fact, the case
$\frac1{\wtr}-\frac1{r}+\frac12(\frac1\wtq-\frac1q)>1 $ can be
obtained by direct summation because $\|T_\delta F\|_{\mixq}\le
C\|F\|_{\mixx {\wtr}{\wtq}}$ for $\delta\ge 1$ . Now by rescaling
it follows that
\begin{equation*}
\Big\|\mathcal P_j\Big(\int_{-\infty}^t \schrd F(s)
ds\Big)\Big\|_{\mixx {r}q}\le
C2^{j(1-\frac12(\frac1\wtr-\frac1r)-(\frac1\wtq-\frac1q))}\|\mathcal
P_jF\|_{\mixx {\wtr}{\wtq}}.
\end{equation*}
Hence we have uniform bounds if
$\frac12(\frac1\wtr-\frac1r)+\frac1\wtq-\frac1q=1$ and the condition
for \eqref{ess} is satisfied. Now note that if
$\frac1\wtr-\frac1r\ge \frac23$, there are $\wtq, q$ satisfying
$\frac12(\frac1\wtr-\frac1r)+\frac1\wtq-\frac1q=1$ and
$\frac1{\wtr}-\frac1{r}+\frac12(\frac1\wtq-\frac1q)\ge 1$.
Therefore, if  $\frac1\wtr-\frac1r\ge \frac23$ and  $1< \wtr< 2<
r<\infty $,  we have for some $1<\wtq<2< q<\infty$
\begin{equation*}
\Big\|\mathcal P_j(\int_{-\infty}^t \schrd F(s) ds)\Big\|_{\mixx
{r}q}\le C\|\mathcal P_jF\|_{\mixx {\wtr}{\wtq}}.
\end{equation*}
This can be put together using Littwood-Paley theorem in $t$. Since
$1< \wtr< 2< r<\infty $ and $1<\wtq\le 2\le q<\infty$, by
Littlewood-paley theorem and Minkowski's inequality
\begin{align*}
\big\|\sum_{j}T(\mathcal P_jF)\big\|_{L_x^rL_t^q}&\lesssim
\big\|\big(\sum_{j}\|T(\mathcal
P_jF)\|_{L_t^q}^2\big)^{1/2}\big\|_{L_x^r}
\lesssim\big(\sum_{j}\|T(\mathcal P_jF)\|_{L_x^rL_t^q}^2\big)^{1/2}\\
&\lesssim\big(\sum_{j}\|\mathcal
P_jF\|_{L_x^{\widetilde{r}^\prime}L_t^{\widetilde{q}'}}^2\big)^{1/2}
\lesssim\big\|\big(\sum_{j}\|\mathcal P_jF\|_{L_t^{\widetilde{q}'}}^2\big)^{1/2}\big\|_{L_x^{\widetilde{r}^\prime}}\\
&\lesssim\|F\|_{L_x^{\widetilde{r}^\prime}L_t^{\widetilde{q}'}}.
\end{align*}
This completes the proof.

\section{Necessary conditions} \label{secsec4}

By constructing some counterexamples, we show the conditions \eqref{q2}, \eqref{q3}.

\subsubsection*{Proof of \eqref{q3}}
Let $M>0$ be a sufficiently large number and let us set
$$\widehat{F}(\xi,\tau)=\varphi(|\xi|)\psi(M^{1/2}(\tau+1)),$$
where $\psi\in\mathcal{S}(\mathbb{R})$ with $\text{supp }\mathcal
F^{-1}(\psi)\in[0,1]$ and $\varphi$ is a smooth function supported
in $(1/2,2)$ with $\varphi(1)=1$. Note that if $|t|\sim M$, then we
may write
$$\int_0^te^{i(t-s)\Delta}F(s)ds=\int_{\mathbb{R}^n}e^{ix\cdot\xi}e^{-it|\xi|^2}\widehat{F}(\xi,-|\xi|^2)d\xi$$
because the support of $F(y,\cdot)$ is contained in $[0,M^{1/2}]$
for all $y$. Since we have
$\int_{S^{n-1}}e^{ix\cdot\xi}d\sigma(\xi)=C|x|^{-\frac{n-2}{2}}J_{\frac{n-2}{2}}(|x|),$
by the asymptotic behavior of Bessel function \cite{St}, we see
that
$$\bigg|\int_0^te^{i(t-s)\Delta}F(s)ds\bigg|\sim
|x|^{-\frac{n-1}{2}}|I(x,t)|$$ for sufficiently large $|x|$, where
\begin{align*}
I(x,t)= \int_0^\infty
r^{-\frac{n-1}{2}}\varphi(r)\psi(M^{1/2}(r^2-1))e^{-itr^2}\cos(r|x|-\pi(n-1)/4)dr.
\end{align*}
We  now set $\widetilde \varphi=r^{-\frac{n-1}{2}}\varphi(r)$. Then
we have
\begin{align*}
I(x,t)&=\int
\widetilde\varphi(1)\psi(M^{1/2}(r^2-1))e^{-itr^2}\cos(r|x|-\pi(n-1)/4)dr\\
&+\int (\widetilde\varphi(r)-
\widetilde\varphi(1))\psi(M^{1/2}(r^2-1))e^{-itr^2}\cos(r|x|-\pi(n-1)/4)dr.
\end{align*}
By the rapid decay of $\psi$, the support of $\widetilde\varphi$,
and the mean value theorem it is easy to see that the second term in
the right hand side is $O(M^{-1})$.  Similarly, for some large
constant $B>0$
\begin{align*}
I(x,t)&=\int_{|r-1|\leq BM^{-1/2}}
\psi(M^{1/2}(r^2-1))e^{-itr^2}\cos(r|x|-\pi(n-1)/4)dr\\
&+O(M^{-1/2}/B^{100})+O(M^{-1}).
\end{align*}
So we get
\begin{align*}
I(x,t)=\frac12\Big(e^{-i\pi(n-1)/4}I_-(x,t)+e^{i\pi(n-1)/4}I_+(x,t)\Big)+O(M^{-1/2}/B^{100}),
\end{align*}
where
$$I_\pm (x,t)=\int_{|r-1|\leq BM^{-1/2}} e^{-i(tr^2\pm r|x|)}\psi(M^{1/2}(r^2-1))dr.$$
By changing variables $r\rightarrow r+1$ and $r\rightarrow
M^{-1/2}r$, it follows that
$$I_-(x,t)=
e^{-i(t-|x|)}M^{-1/2}\int_{|r|\leq B}
e^{-i\big(tM^{-1}r^2+(2t-|x|)M^{-1/2}r\big)}\psi(2r+M^{-1/2}r^2)dr.$$
So, if $|t|\sim M/B^2$ and $|2t-|x||\lesssim M^{1/2}/B$, we get
$|I_-(x,t)|\gtrsim M^{-1/2}.$
On the other hand, we see $|I_+(x,t)|\lesssim B^2M^{-1}$ if $|t|\sim M$ and $|x|\sim M$.
Consequently, if $|t|\sim M$, $|x|\sim M$ and $|2t-|x||\lesssim M^{1/2}$, then $|I(x,t)|\gtrsim M^{-1/2}$.
Therefore, we see that
$$\bigg\|\int_0^te^{i(t-s)\Delta}F(s)ds\bigg\|_{L_x^rL_t^q}
\gtrsim M^{-n/2}M^{1/2q}M^{n/r}.$$
Also, it is easy to see that
$\|F\|_{L_x^{\widetilde{r}'}L_t^{\widetilde{q}'}}
\lesssim M^{-1/2}M^{1/2\widetilde{q}'}.$ Hence, the estimate \eqref{iinhomo} implies
that
\[ M^{-n/2}M^{1/2q}M^{n/r}\lesssim M^{-1/2}M^{1/2\widetilde{q}'}.\]
By letting $M\rightarrow\infty$, we get the  first inequality in \eqref{q3},
and the second one follows from duality.
\qed

\subsubsection*{Proof of \eqref{q2}}
Let us denote
\[U(F)(x,t)=\int_0^t \schrd F(s) ds.\]
Then, using the kernel of $e^{it\Delta}$,  we have
\[U(F)(x,t)=\int_0^t\int_{\mathbb{R}^n}(t-s)^{-n/2}e^{\frac{i|x-y|^2}{4(t-s)}}F(y,s)dyds.\]
For $0<\delta\ll1$, we set
$$F(y,s)=\Phi(\delta^{1/2}(y_1+2s),\delta^{1/2}\overline{y},\delta s)e^{-i(y_1+s)},$$
where
$\Phi(y_1,\overline{y},s)=\chi(y_1)\chi(y_2)\cdot\cdot\cdot\chi(y_n)\chi(s)$ and $\chi=\chi_{[0,1]}$.
By the change of variables $y_1\rightarrow y_1-2s$, we see that
$$e^{i(x_1+t)}U(F)(x,t)=\int_0^t\int_{\mathbb{R}^n}(t-s)^{-n/2}e^{iP(x,y,t,s)}
\Phi(\delta^{1/2}y_1,\delta^{1/2}\overline{y},\delta s)dyds,$$
where
$$P(x,y,t,s)=\frac{|\overline{x}-\overline{y}|^2+(x_1-y_1+2t)^2}{4(t-s)}.$$
Note that $|P(x,y,t,s)|\lesssim1$
if $(x_1+2t)^2\leq\delta^{-1}$, $|\overline{x}|\leq\delta^{-1/2}$
and $100\delta^{-1}\leq t\leq200\delta^{-1}$.
So we see
\begin{align*}
|U(F)(x,t)|&\gtrsim\delta^{n/2}\bigg|\int_0^t\int_{\mathbb{R}^n}\Phi dyds\bigg|\gtrsim\delta^{-1},
\end{align*}
provided that $(x_1+2t)^2\leq\delta^{-1}$, $|\overline{x}|\leq\delta^{-1/2}$
and $100\delta^{-1}\leq t\leq200\delta^{-1}$. Hence
$$\|U(F)\|_{L_x^rL_t^q}\gtrsim\delta^{-1}\delta^{-1/2q}\delta^{-(n+1)/2r}.$$
On the other hand, $\|F\|_{L_x^{\widetilde{r}^\prime}L_t^{\widetilde{q}^\prime}}
\leq C\delta^{-1/2\widetilde{q}^\prime}\delta^{-(n+1)/2\widetilde{r}^\prime}$. From \eqref{iinhomo}
we get
\[\delta^{-1}\delta^{-1/2q}\delta^{-(n+1)/2r}\lesssim \delta^{-1/2\widetilde{q}^\prime}\delta^{-(n+1)/2\widetilde{r}^\prime}.\]
By letting $\delta\rightarrow0$, we  get \eqref{q2}.
\qed

\section*{Acknowledgment}
The first named author was supported in part by NRF grant
2012-008373 (Republic of Korea).


\end{document}